\documentclass[a4paper,12pt]{article}

\usepackage{geometry}
\usepackage[latin1]{inputenc}
\usepackage[T1]{fontenc}

\usepackage{graphicx} 

\usepackage{amsmath,amsthm,amsfonts,amssymb,amsxtra}

\usepackage[all]{xypic}

\usepackage{array}

\newcommand{\produitscalaire}[2]{\langle{#1}\,,\,{#2}\rangle}

\theoremstyle{plain}                    

\newtheorem{theoreme}{Theorem}
\newtheorem*{theoreme*}{Theorem}
\newtheorem{corollaire}{Corollary}
\newtheorem{lemme}{Lemma}

\theoremstyle{definition} 

\newtheorem*{remarque}{Remark}

\title{On Gaussian Brunn-Minkowski inequalities}
\author{Franck Barthe and Nolwen Huet}
\date{\today}

\begin{document}
\maketitle

\begin{abstract}
   In this paper, we are interested in Gaussian versions of the classical
Brunn-Minkowski inequality. We prove in a  streamlined way a semigroup version
of the Ehrard inequality for $m$ Borel or convex sets based on a previous work
by Borell. Our method also  allows us  to have semigroup proofs of the geometric
Brascamp-Lieb inequality and of the reverse one which follow exactly the same
lines.
\end{abstract}

\noindent{\em 2000 Mathematics Subject Classification}: 60E15, 60G15, 52A40, 35K05.

\noindent{\em Keywords}: Brunn-Minkowski, Gaussian measure, Heat equation, Brascamp-Lieb inequalities.

\section{Introduction}

In this paper, we are interested in Gaussian versions of the classical Brunn-Minkowski inequality on the Lebesgue measure 
of sum-sets (see e.g. \cite{pisiVCBB,schnBMT}). On $\mathbb R^n$ with its canonical Euclidean structure $(\produitscalaire{\cdot}{\cdot},|\cdot|)$
we consider the standard Gaussian measure $\gamma_n(dx)= (2\pi)^{-n/2} \exp(-|x|^2/2) \, dx$, $x\in \mathbb R^n$.
Given $\alpha,\beta\in \mathbb R$ and sets $A,B\subset \mathbb R^n$, we recall that their Minkowski combination is defined by
  $$ \alpha A+\beta B =\{ \alpha a+\beta b;\; (a,b)\in A\times B \}.$$
Using symmetrization techniques, Ehrhard  \cite{EhrhardGBM} proved a sharp lower bound on  the Gaussian measure
 of a convex combination of convex sets.
Namely: if $\alpha,\beta\ge 0$ satisfy $\alpha+\beta=1$ and if $A,B\subset \mathbb R^n$ are convex, then
\[
\Phi^{-1}\circ\gamma_n(\alpha A+\beta B)\geq \alpha \Phi^{-1}\circ\gamma_n(A)+\beta\Phi^{-1}\circ\gamma_n(B),
\]  
where  $\Phi$ is the cumulative distribution function of $\gamma_1$. This inequality becomes an equality when 
$A$ and $B$ are parallel half-spaces or the same convex set. Lata\l{}a  \cite{LatalaGBM} showed that the inequality remains valid when $A$
is convex and $B$ is an arbitrary Borel set. In the remarkable paper \cite{BorellGBM}, Borell was able to remove the
remaining convexity assumption. He actually derived a functional version of the inequality (in the spirit of the 
Prékopa-Leindler inequality) by a wonderful interpolation technique based on the heat equation. In a series of papers,
Borell extended the inequality to more general combinations:

\begin{theoreme*}[Borell \cite{BorellGBMm}] Let $\alpha_1,\ldots,\alpha_m>0$. The  inequality 
\begin{equation}\label{gbmi} 
\Phi^{-1}\circ\gamma_n\Big({\textstyle\sum}\alpha_i A_i\Big)\geq  \sum \alpha_i \Phi^{-1}\circ\gamma_n(A_i)
\end{equation}
 holds for all Borel sets $A_1,\ldots,A_m$  in $\mathbb{R}^n$ if and only if 
\begin{equation*}
\sum \alpha_i \geq 1 \quad \text{ and } \quad \forall j,\ \alpha_j-\sum_{i\neq j}\alpha_i \leq 1.
\end{equation*}
Moreover, it holds for all convex sets $A_1,\ldots,A_m$ in $\mathbb{R}^n$ if and only if 
\[
\sum \alpha_i \geq 1.
\]
\end{theoreme*}
Borell  established  the case $m=2$ for Borel sets in \cite{BorellGBM2} thanks to his semigroup argument.
 His proof in \cite{BorellGBMm} of the general case relies on a tricky and somewhat complicated induction.
Remark that a linear combination of Borel sets need not be a Borel set; however it is analytic or Suslin, hence universally
measurable, see e.g. \cite{fedeGMT}.

In this note we give a slight extension of the above statement (which can actually be derived directly from the theorem of Borell,  as pointed out by the referee). More importantly we propose a streamlined version
of the semigroup argument for $m$ functions directly, which allows to take advantage of convexity type assumptions.
This better understanding of the semigroup technique also allows to study more general situations. 
The main result is stated next. It involves the heat semigroup, for which we recall the definition: given a Borel nonnegative function $f$ on $\mathbb R^n$, its evolute at time $t\ge 0$ is the function $P_tf$ given by
\[
P_t f(x)=\int f\big(x+ \sqrt{t} \, y\big) \,\gamma_n(dy)=\mathrm{E}\big(f(x+B_t)\big)
\]
where $B$ is an $n$-dimensional Brownian motion. By convention $\infty-\infty=-\infty$ so that inequalities like Inequality \eqref{gbmi}, or the one introduced in the next theorem, make sense.
 
\begin{theoreme}\label{main_theorem}
Let $I_{conv}\subset\{1,\ldots,m\}$,  $\alpha_1,\ldots,\alpha_m>0$. 
The following assertions are equivalent: 
\begin{enumerate}
\item The parameter $\alpha$ satisfies
\begin{equation}\label{alpha}
\sum \alpha_i \geq 1 \quad \text{ and } \quad \forall j\notin I_{conv},\ \alpha_j-\sum_{i\neq j}\alpha_i \leq 1.
\end{equation}
\item For all Borel sets $A_1,\ldots,A_m$ in $\mathbb{R}^n$ such that  $A_i$ is convex when $i\in I_{conv}$,
\begin{equation*}
\Phi^{-1}\circ\gamma\Big({\textstyle\sum}\alpha_i A_i\Big)\geq  \sum \alpha_i \Phi^{-1}\circ\gamma(A_i)
\end{equation*}
\item For all Borel functions $h,f_1,\ldots,f_m$ from $\mathbb{R}^n$ to  $[0,1]$ such that $\Phi^{-1} \circ f_i$ is concave  when $i\in I_{conv}$, if
\begin{equation*}
\forall x_1,\ldots, x_m \in\mathbb{R}^n,\quad \Phi^{-1}\circ h\big({\textstyle\sum}\alpha_i x_i\big)\geq \sum \alpha_i \Phi^{-1}\circ f_i(x_i),
\end{equation*}
then
\begin{equation*}
\Phi^{-1} \left( \int h \, d\gamma \right) \geq \sum \alpha_i  \Phi^{-1}\left( \int f_i \, d\gamma \right).
\end{equation*}
\item For all Borel functions $h,f_1,\ldots,f_m$ from $\mathbb{R}^n$ to  $[0,1]$ such that $\Phi^{-1} \circ f_i$ is concave  when $i\in I_{conv}$, if
\begin{equation*}
\forall x_1,\ldots,x_m \in\mathbb{R}^n,\quad \Phi^{-1}\circ h\big({\textstyle\sum}\alpha_i x_i\big)\geq \sum \alpha_i \Phi^{-1}\circ f_i(x_i),
\end{equation*}
then for all $t\ge 0$
\begin{equation*}
\forall x_1,\ldots, x_m
 \in\mathbb{R}^n,\quad \Phi^{-1}\circ P_t h\big({\textstyle\sum}\alpha_i x_i\big)\geq \sum \alpha_i \Phi^{-1}\circ P_t f_i(x_i).
\end{equation*}
\end{enumerate}
\end{theoreme}

\begin{remarque}
   Condition \eqref{alpha} can be rephrased as 
$$\displaystyle \sum \alpha_i \ge \max\big(1, \max\{ 2 \alpha_j-1;\;j\not\in I_{conv} \}\big).$$
 Actually the condition will come up in our argument
in the following geometric form: there exist vectors $u_1,\ldots,u_m \in \mathbb R^m$ such that for all $i\in I_{conv}$, $|u_i|\le 1$,
for all $i\not\in I_{conv}$, $|u_i|=1$, and $|\sum \alpha_i u_i|=1$.
\end{remarque}

In the next section we show that the condition on $\alpha$ implies the fourth (and formally strongest) assumption in
the latter theorem, when restricted to smooth enough functions. The third section completes the proof of the theorem.
In the final section we discuss related problems.

Before going further, let us introduce some notation. 
\begin{itemize}
\item   We consider functions depending on a time variable $t$ and a space variable $x$. The time derivative is denoted by $\partial_t$, while the gradient, Hessian, and Laplacian in $x$ are denoted by $\nabla_x$, $\mathrm{Hess\,}_x$, and $\Delta_x$, omitting the index $x$ when 
there is no ambiguity.
\item The unit Euclidean (closed) ball and sphere of $\mathbb{R}^d$ are denoted respectively by $\mathbb{B}^d$ and $\mathbb{S}^{d-1}$.
\item For  $A\subset\mathbb{R}^d$, we set $A^\varepsilon=A + \varepsilon \mathbb{B}^d$. The notation $A_i^\varepsilon$ means $(A_i)^\varepsilon$.
\end{itemize}

\section{Functional and semigroup approach}

As already mentioned we  follow Borell's semigroup approach of the Gaussian Brunn-Minkowski inequalities
 (see  \cite{BorellGBM} and  \cite{BorellGBM2}): for parameters $\alpha$ verifying \eqref{alpha}, 
 the plan is two show the functional version of the inequality (the third
assertion of Theorem~\ref{main_theorem}), by means of the heat semigroup. Note that the fourth assertion implies the third
one when choosing $t=1$ and $x_i=0$ in the last equation. So our aim is to establish the fourth assumption. More precisely,
given
Borel functions  $h,f_1,\ldots,f_m$ from  $\mathbb{R}^n$ taking into  $(0,1)$, 
 we define  $C$ on $[0,T]\times(\mathbb{R}^n)^m$ by
\begin{equation*}
C(t,x)=C(t,x_1,\ldots,x_m)= \Phi^{-1}\circ P_t h\big({\textstyle\sum}\alpha_i x_i\big) - \sum \alpha_i \Phi^{-1}\circ P_t f_i(x_i).
\end{equation*}
Since $P_0f=f$ the assumption 
\begin{equation}\label{fgbmi0}
\forall x_i\in\mathbb{R}^n,\quad \Phi^{-1}\circ h\big({\textstyle\sum}\alpha_i x_i\big)\geq \sum \alpha_i \Phi^{-1}\circ f_i(x_i)
\end{equation}
translates as $C(0,\,.\,) \ge 0$. 
Our task is  to prove 
\begin{equation*}\label{Cgbmi}
C(0,\,.\,)\geq0 \quad \Longrightarrow \quad \forall t\geq0,
\ C(t,\,.\,)\geq 0.
\end{equation*}

\subsection{Preliminaries}

When the functions $h$ and $f_i$ are smooth enough, the time evolution of $P_th$ and $P_tf_i$ is described by the heat equation.
This yields a  differential equation satisfied by $C$. Our problem boils down to determine whether this evolution equation
preserves  nonnegative functions. This is clearly related to  the maximum principle for parabolic equations (see e.g. \cite{brezis}).
We will use the following  lemma.

\begin{lemme}\label{lemma:equadiff}
Assume that  $C$ is twice differentiable. If
\begin{equation}\label{equadiff}
\left\{
\begin{array}{l}
\mathrm{Hess} (C) \geq 0\\
\nabla C=0\\
C\leq 0
\end{array}
\right.
\quad\Longrightarrow \quad
\partial_t C \geq 0
\end{equation}
and if for some $T>0$
\begin{equation}\label{infinipositif}
\liminf_{\left| x \right|\to\infty} \left(\inf_{0\leq t \leq T} C(x,t)\right)  \geq 0,
\end{equation}
then
\begin{equation*}
C(0,\,.\,)\geq0 \quad \Longrightarrow \quad \forall t\in[0,T],
\ C(t,\,.\,)\geq 0.
\end{equation*}
\end{lemme}


\begin{proof} For $\varepsilon>0$, set $C_\varepsilon(t,x)=C(t,x)+\varepsilon t$ on $[0,T]\times(\mathbb{R}^n)^m$. If $C_\varepsilon<0$ at some point, then $C_\varepsilon$ reaches its minimum at a point $(t_0,x_0)$ where $\nabla C =0$, $\mathrm{Hess}{} (C)\geq 0$, $C<0$, and $\partial_t C + \varepsilon \leq 0$ ($=0$ if $t_0<T$). By the hypotheses, it implies $\partial_t C \geq 0$ which is in contradiction with $\partial_t C \leq -\varepsilon$. So for all $\varepsilon>0$ and $T>0$, $C_\varepsilon$ is non-negative on $[0,T]\times(\mathbb{R}^n)^m$, thus $C$ is non-negative everywhere.
\end{proof}

Property \eqref{infinipositif} is true under mild assumptions on $h$ and $f_i$ which are related to the initial condition $C(0,\,.\,)\geq0$ in the large: 
\begin{lemme}\label{lem:liminf}
If there exist $a_1,\ldots,a_m\in \mathbb{R}$ such that
\begin{itemize}
\item 
$\displaystyle\limsup_{\left| x \right|\to\infty} f_i(x) \leq \Phi(a_i)$  
\item $ h \geq \Phi\big({\textstyle\sum} \alpha_i a_i\big)$
\end{itemize}
then for all $T>0$,
\begin{equation*}
\liminf_{\left| x \right|\to\infty} \left(\inf_{0\leq t \leq T} C(x,t)\right)  \geq 0.
\end{equation*}
\end{lemme}

\begin{proof}
Let $\delta>0$. By continuity of $\Phi^{-1}$, there exists $\varepsilon>0$ such that
\begin{equation*}
\Phi^{-1}\big(\Phi(a_i)+2\varepsilon\big)\leq a_i + \frac{\delta}{\sum \alpha_j}.
\end{equation*}
Let $r>0$ be such that 
$
\gamma_n\left(r\mathbb{B}^n\right)= 1-\varepsilon.
$
Then, for 
$0\leq t\leq T$,
\begin{align*}
P_t f_i(x_i)&=\int_{r\mathbb{B}^n} f_i(x_i+ \sqrt{t} \, y) \,\gamma_n(dy)
+\int_{\left(r\mathbb{B}^n\right)^\complement} f_i(x_i+ \sqrt{t} \, y) \,\gamma_n(dy)\\
&\leq (1-\varepsilon) \sup_{x_i+r\sqrt{t}\,\mathbb{B}^n}f_i + \varepsilon \sup f_i\\
&\leq  \sup_{x_i+r\sqrt{T}\,\mathbb{B}^n}f_i + \varepsilon \\
&\leq \Phi(a_i)  + 2\varepsilon \quad \text{for $\left| x_i \right|$ large enough}.
\end{align*}
Moreover
$P_t h\geq \Phi\big({\textstyle\sum}\alpha_i a_i\big)$
so  for $\left| x \right|$ large enough and  for $0\leq t \leq T$, it holds 
$
C(t,x)\geq  -\delta.
$
As $\delta>0$ was arbitrary, the proof is complete.
\end{proof}

Checking Property \eqref{equadiff} of Lemma \ref{lemma:equadiff}  requires the  following lemma:
\begin{lemme}\label{lemma:imphiconcave}
Let $d\geq2$, $\alpha_1,\ldots,\alpha_m>0$. Let $k$ be an integer with  $0\leq k \leq m$ and
\[
\varphi: 
\begin{array}[t]{ccc}
(\mathbb{S}^{d-1})^k \times (\mathbb{B}^{d})^{m-k} & \to & \mathbb{R}_+\\
(v_1,\ldots,v_m)\phantom{^{m-k}}&\mapsto&\left| \sum \alpha_{i} v_i \right|
\end{array}
.
\] 
Then the image of $\varphi$ is the interval
\[
J:=
\Bigg[
\max\bigg( \Big\{0\Big\}\cup \Big\{ \alpha_j - \sum_{i\neq j}\alpha_i, 1\leq j\leq k \Big\}\bigg)
 \, ,\, \sum\alpha_i
 \Bigg].
\]
\end{lemme}

\begin{proof}
As $\varphi$ is continuous on a compact connected set,  $\mathrm{Im}(\varphi)=[\min\varphi,\max\varphi]$.
Plainly $|\sum \alpha_i v_i|\le \sum \alpha_i,$ with equality if $v_1=\cdots=v_m$ is a 
unit vector. So  $\max\varphi =\sum_i \alpha_i$.
For all $j\le k$,  since $|v_j|=1$,  the triangle inequality gives
\[
 \left| \sum \alpha_i v_i \right| \ge   \alpha_j |v_j|- \sum_{i\neq j} \alpha_i|v_i| \ge \alpha_j-  \sum_{i\neq j} \alpha_i.
\]
Hence  $\mathrm{Im}(\varphi)\subset J$ and these two segments have the same upper bound. Next we deal with the lower bound.
 Let us consider a point $(v_1,\ldots,v_m)$  where $\varphi$ achieves its minimum, and differentiate:

 For $j\le k$, $v_j$ lies in the unit sphere.  Applying Lagrange multipliers theorem to $\varphi^2$ with respect to $v_j$  gives a 
 real number $\lambda_j$ such that, 
\begin{equation}\label{eq:lagrange}
\alpha_j\sum_i \alpha_i v_i
=
\lambda_j
v_j.
\end{equation}

  For $j>k$, the $j$-th variable lives in $\mathbb B^d$. If $|v_j|<1$ the minimum is achieved at an interior point and
  the full gradient on $\varphi^2$ with respect to the $j$-th variable is zero. Hence $\sum_i \alpha_i v_i=0$.
  On the other hand if at the minimum $|v_j|=1$, differentiating in the $j$-th variable only along the unit sphere  gives again
  the existence of $\lambda_j\in \mathbb R$ such that \eqref{eq:lagrange} is verified.
    
Eventually, we face 2 cases:
\begin{enumerate}
\item Either $\sum \alpha_i v_i=0$ and $\min \varphi=0$. In this case, the triangle inequality gives $0=|\sum \alpha_i v_i|\ge
   \alpha_j-\sum_{i\neq j} \alpha_i$ whenever $j\le k$. 
\item Or the $v_i$'s are  colinear  unit vectors and
 there exists a partition ${S_+}\cup
 {S_-} = \{1,\ldots,m\}$ and a unit vector $v$ such that
\[ 
\min\varphi= \Big|\sum_{S_+} \alpha_i v - \sum_{S_-} \alpha_i v \Big|=\sum_{S_+} \alpha_i - \sum_{S_-} \alpha_i >0.
\] 
Assume that ${S_+}$ contains 2 indices $j$ and $\ell$. Let $e_1$ and $e_2$ be 2 orthonormal vectors of $\mathbb{R}^d$ and let us denote by $R(\theta)$ the rotation in the plane $\mathrm{Vect} (e_1,e_2)$ of angle $\theta$.
The length of the vector $\alpha_j R(\theta)e_1+\alpha_\ell e_1$ is a decreasing and continuous function of $\theta\in [0,\pi]$.
Denote by $U(\theta)$ the rotation in the plane  $\mathrm{Vect} (e_1,e_2)$ which maps this vector to $|\alpha_j R(\theta)e_1+\alpha_\ell e_1|e_1$.
Then
$$ \alpha_j U(\theta)R(\theta)e_1+\alpha_\ell U(\theta)e_1+\sum_{{S_+}\setminus\{j,\ell\}}\alpha_i e_1 -\sum_{S_-} \alpha_i e_1
 = \lambda(\theta) e_1,$$
where $\lambda(0)=\sum_{S_+}\alpha_i-\sum_{S_-}\alpha_i=\min \varphi>0$ and $\lambda$ is continuous 
and decreasing in $\theta \in [0,\pi]$. This contradicts the minimality of $\min \varphi$.
 So ${S_+}$ contains a single index $j$ and
 $$\min\varphi=\Big|\alpha_j v - \sum_{i\neq j} \alpha_i v \Big|=\alpha_j - \sum_{i\neq j} \alpha_i>0.$$
Note that necessarily $j\le k$, otherwise one could get a shorter vector by replacing $v_j=v$ by $(1-\varepsilon)v$.
Besides, the condition $\alpha_j - \sum_{i\neq j} \alpha_i>0$ ensures that 
  $\alpha_j>\alpha_\ell$ for $\ell\neq j$. This implies that  for $\ell \neq j$, 
 \[
 \alpha_\ell - \sum_{i\neq \ell} \alpha_i \leq \alpha_\ell-\alpha_j< 0 <\alpha_j - \sum_{i\neq j} \alpha_i.
 \]
 So $\min \varphi=\max\bigg( \Big\{0\Big\}\cup \Big\{ \alpha_j - \sum_{i\neq j}\alpha_i, 1\leq j\leq k \Big\}\bigg)$
 as claimed.
\end{enumerate}
\end{proof}

\subsection{Semigroup proof for smooth functions}\label{sec:thmfunc}
We deal with smooth functions first, in order to ensure that $P_t f_i$ and $P_th$ verify the heat equation.
 This restrictive assumption   will be removed in  Section \ref{sec:sets} where the proof of Theorem \ref{main_theorem}
is completed.
\begin{theoreme}\label{thm:principal}
Let $f_i, i=1,\ldots,m$, and $h$ be twice continuously differentiable functions from  $\mathbb{R}^n$ to $(0,1)$ satisfying the hypotheses of Lemma \ref{lem:liminf}.
 Assume moreover that for $f=f_i \text{ or }h$,
\begin{equation*}
\forall t>0, \forall x\in\mathbb{R}^n,\quad \left|{\nabla f(x+\sqrt{t}\,y)}\right| e^{-\frac{\left| y \right|^2}{2}}\xrightarrow[\left| y \right|\to\infty]{} 0.
\end{equation*}
Let $\alpha_1,\ldots,\alpha_m$ be positive real numbers such that
\begin{equation*}
\sum \alpha_i \geq 1 \quad \text{ and } \quad \forall j,\ \alpha_j-\sum_{i\neq j}\alpha_i \leq 1.
\end{equation*}
If
\begin{equation*}
\forall x_i\in\mathbb{R}^n,\quad  \Phi^{-1}\circ h\big({\textstyle\sum}\alpha_i x_i\big)\geq \sum \alpha_i \Phi^{-1}\circ f_i(x_i),
\end{equation*}
then
\begin{equation*}
\forall t\geq0,\, \forall x_i\in\mathbb{R}^n, \quad  \Phi^{-1}\circ P_t h\big({\textstyle\sum}\alpha_i x_i\big)\geq \sum \alpha_i  \Phi^{-1}\circ P_t f_i(x_i).
\end{equation*}
\end{theoreme}

\begin{proof}
Let us recall that $C$ is defined by
\begin{equation*}
C(t,x)=C(t,x_1,\ldots,x_m)= H\big(t,{\textstyle\sum\alpha_i x_i}\big)-\sum \alpha_i F_i(t,x_i)
\end{equation*}
where we have set
\begin{equation*}
H(t,y)=\Phi^{-1}\circ P_t h(y)\quad \text{and} \quad F_i(t,y)=\Phi^{-1}\circ P_t f_i(y).
\end{equation*}
In what follows, we omit the variables and write $H$ for   $H\big(t,\sum \alpha_i x_i\big)$  and $F_i$ instead of $F_i(t,x_i)$.
With this simplified notation,
\[
\begin{array}{lc>{\displaystyle}l<{\text{\Large\strut}}}
C &=& H-\sum \alpha_i  F_i,\\
\nabla_{x_i} C &=& \alpha_i (\nabla H - \nabla F_i), \\
\nabla_{x_i}\nabla^*_{x_j}C&=&\alpha_i \alpha_j \mathrm{Hess} (H) - \delta_{i j} \alpha_i \mathrm{Hess} (F_i).\\
\end{array}
\]
Moreover, one can use  the property of heat kernel to derive a differential equation for $F_i$ and $H$. Indeed, for any $f$ satisfying hypotheses of the theorem, we can perform an integration by parts so that it holds
\begin{equation*}
\partial_t P_t f = \frac{1}{2} \Delta P_t f.
\end{equation*} 
Then we set $F=\Phi^{-1}\circ P_t f$ and use the identity 
 $(1/\Phi'(x))'=x/\Phi'(x)$ to show \label{equadiff_F}
\[
\begin{array}{lc>{\displaystyle}l<{\text{\Huge\strut}}}
\partial_t F&=& \frac{\partial_t P_t f}{\Phi'(F)} \ =\ \frac{\Delta P_t f}{2\,\Phi'(F)},\\
\nabla F &=& \frac{\nabla P_t f}{\Phi'(F)}, \\
\Delta F&=&\frac{\Delta P_t f}{\Phi'(F)} + F \frac{\left| \nabla P_t f \right|^2}{(\Phi'(F))^2}.\\
\end{array}
\]
We put all together to get
\begin{equation*}
\partial_tF=\frac{1}{2} \left( \Delta F - F\left| \nabla F \right|^2 \right)
\end{equation*}
and to deduce the following differential equation for $C$:
\begin{equation*}\label{eq:C}
\partial_tC=\frac{1}{2} (\mathcal{S}  + \mathcal{P}  )
\end{equation*}
where the second order part is 
\begin{equation*}\label{eq:epsilon}
\mathcal{S} =  \Delta H - \sum \alpha_i \Delta F_i
\end{equation*}
and the terms of lower order are
\begin{equation*}\label{eq:lower_order}
\mathcal{P}  = - \left( H\left| \nabla H \right|^2 - \sum \alpha_i F_i\left| \nabla F_i \right|^2\ \right) .
\end{equation*}
We will conclude using Lemma \ref{lemma:equadiff}. So we need to check Condition~\eqref{equadiff}.
First we note that $\mathcal{P} $ is non-negative when  $\nabla C = 0$ and $C\leq0$, regardless of $\alpha$. Indeed, $\nabla C = 0$ implies that $\nabla F_i=\nabla H$ for all $i$. So $\mathcal{P} = -\left| \nabla H \right|^2 C$ which is non-negative if $C\leq0$.

It remains to deal with the second order part. It is enough to express $\mathcal{S}$ as $\mathcal{E}C$ for some elliptic operator
 $\mathcal{E}$, since then  $\mathrm{Hess} (C)\geq0$ implies  $\mathcal{S} \geq 0$.
Such a second order operator can be written as $\mathcal{E}=\nabla^* A \nabla$ where $A$ is a symmetric matrix $nm \times nm$. Moreover
 $\mathcal{E}$ is elliptic if and only if $A$ is positive semi-definite. 
In view of the structure of the problem, it is natural to look for matrices of the following block  form
$$ A= B\otimes I_n =(b_{i j} I_n)_{1\leq i,j \leq m}\, ,$$
where $I_n$ is the identity $n\times n$ matrix and $B$ is a positive semi-definite matrix of size $m$.
Denoting $x_i=(x_{i,1},\ldots,x_{i,n})$, 
\begin{align*}
   \mathcal E C&= \sum_{i,j=1}^{m} b_{i,j} \left( \sum_{k=1}^{n} \frac{\partial^2}{\partial x_{i,k}\partial x_{j,k}}C\right) 
     =  \sum_{i,j=1}^{m} b_{i,j} \big(\alpha_i\alpha_j \Delta H -\delta_{i,j} \alpha_i \Delta F_i\big)\\ 
   &= \produitscalaire{\alpha}{B\alpha} \Delta H -\sum_{i=1}^m b_{i,i} \alpha_i \Delta F_i.
\end{align*}
Hence there exists an elliptic operator $\mathcal E$ of the above form such that 
$\mathcal E C= \mathcal S= \Delta H-\sum_{i=1}^{m} \alpha_i \Delta F_i$ if there exits a positive semi-definite matrix $B$
of size $m$ such that 

\[
\produitscalaire{\alpha}{B\alpha}=\produitscalaire{e_1}{Be_1}=\cdots=\produitscalaire{e_m}{Be_m}=1
\] 
where $(e_i)_i$ is the canonical basis of $\mathbb{R}^m$.
Now a positive semi-definite  matrix $B$ can be decomposed into $B=V^*V$ where $V$ is a square matrix of size $m$. Calling $v_1,\ldots,v_m\in\mathbb{R}^m$ the columns of $V$, 
 we can translate the latter  into conditions on vectors $v_i$. 
Actually, we are looking for vectors  $v_1,\ldots, v_m \in \mathbb{R}^m$  with 
\[\left| v_1 \right|=\cdots=\left| v_m \right|=\left| \sum\alpha_i v_i \right|=1.\]
By Lemma \ref{lemma:imphiconcave} for $k=m$, this is possible exactly when $\alpha$ satisfies the claimed condition:
\begin{equation*}
\sum \alpha_i \geq 1 \quad \text{ and } \quad \forall j,\ \alpha_j-\sum_{i\neq j}\alpha_i \leq 1.
\end{equation*}
\end{proof}
The following corollary will be useful in the next section.
\begin{corollaire}\label{cor:concave}
Let $f$ be a function on $\mathbb{R}^n$ taking values in $(0,1)$  and vanishing at infinity, i.e. $\lim_{|x|\to\infty}f(x)= 0$. 
Assume also that
\begin{equation*}
\forall t>0, \forall x\in\mathbb{R}^n,\quad \left|{\nabla f(x+\sqrt{t}\,y)}\right| e^{-\frac{\left| y \right|^2}{2}}\xrightarrow[\left| y \right|\to\infty]{} 0.
\end{equation*}
If $\Phi^{-1}\circ f$ is concave, then $\Phi^{-1}\circ P_t f$ is concave for all $t\geq0$.
\end{corollaire}
\begin{proof}
Let $1>\varepsilon >0$ and 
 $\alpha_i>0$ with $\sum\alpha_i =1$. Choosing $h=\varepsilon + (1-\varepsilon) f  \ge f$ and $f_i=f$ for $i\geq 1$, 
one can check that the latter theorem applies. Hence  for all $t\geq0$ and $x_i\in\mathbb{R}^n$:
\[
\Phi^{-1}\circ P_t (\varepsilon + (1-\varepsilon) f )\big({\textstyle\sum}\alpha_i x_i\big)\geq \sum \alpha_i  \Phi^{-1}\circ P_t f(x_i).
\]
Letting $\varepsilon$ go to 0, we get by monotone convergence  that $\Phi^{-1}\circ P_t f$ is concave.
\end{proof}


\subsection{$\Phi^{-1}$-concave functions}\label{concave_functions}
 When some  of the $f_i$'s  are $\Phi^{-1}$-concave,  the conditions on the  parameters can be relaxed.
Such functions allow  to approximate characteristic functions of convex sets. They will be useful in Section \ref{sec:sets}.
\begin{theoreme}\label{theorem:concave}
Let $I_{conv} \subset \{1,\ldots,m\}$. 
Let $f_i, i=1,\ldots,m$, and $h$ be twice continuously differentiable functions from  $\mathbb{R}^n$ to $(0,1)$ satisfying the hypotheses of Lemma \ref{lem:liminf}.
 Assume also that for $f=f_i \text{ or }h$,
\begin{equation*}
\forall t>0, \forall x\in\mathbb{R}^n,\quad \left|{\nabla f(x+\sqrt{t}\,y)}\right| e^{-\frac{\left| y \right|^2}{2}}\xrightarrow[\left| y \right|\to\infty]{} 0.
\end{equation*}
Assume  moreover that $\Phi^{-1}\circ f_i$ is concave, decreasing towards $-\infty$ at infinity for all $i\in I_{conv}$.

Let $\alpha_1,\ldots,\alpha_m$ be positive  numbers satisfying
\begin{equation*}
\sum \alpha_i \geq 1 \quad \text{ and } \quad \forall j\notin I_{conv},\ \alpha_j-\sum_{i\neq j}\alpha_i \leq 1.
\end{equation*}
If
\begin{equation*}
\forall x_i\in\mathbb{R}^n,\quad \Phi^{-1}\circ h\big({\textstyle\sum}\alpha_i x_i\big)\geq \sum \alpha_i \Phi^{-1}\circ f_i(x_i),
\end{equation*}
then
\begin{equation*}
\forall t\geq0,\, \forall x_i\in\mathbb{R}^n, \quad \Phi^{-1}\circ P_t h\big({\textstyle\sum}\alpha_i x_i\big)\geq \sum \alpha_i  \Phi^{-1}\circ P_t f_i(x_i).
\end{equation*}
\end{theoreme}
\begin{proof}
As in the proof of Theorem \ref{thm:principal}, we try to apply Lemma \ref{lemma:equadiff} to the equation satisfied by $C$:
\begin{equation*}
\partial_tC(t,x)=\frac{1}{2} (\mathcal{S}  + \mathcal{P}  ).
\end{equation*}
We have already shown that $\mathcal{P} $ is non-negative when  $\nabla C = 0$ and $C\leq0$, for any $\alpha_1,\ldots,\alpha_m$.
 We would like to prove that the conditions on $\alpha$ in the theorem imply  that  $\mathcal{S} $ is 
non-negative whenever $\mathrm{Hess}(C)\geq 0$.

By Corollary \ref{cor:concave}, for all  $i\in I_{conv}$ the function $F_{i}$ is concave, hence $\Delta F_i \leq 0$.
 So we are done if we can write
\[
\mathcal{S}  = \mathcal{E} C - \sum_{i\in I_{conv}}\lambda_i \Delta F_i,
\]
for some elliptic operator  $\mathcal{E}$  and some $\lambda_i\geq0$ . 
As in the proof of the previous theorem, we are looking  for operators of the form  $\mathcal{E}=\nabla^* A \nabla$
with $A=B\otimes I_n=(b_{i j} I_n)_{1\leq i,j \leq m}$ where $B$ is a symmetric positive semi-definite matrix $m\times m$.
Hence our task is to find $B\ge 0$ and $\lambda_i\ge 0$ such that  $\lambda_i=0$ when $i\notin I_{conv}$ and  
\[
 \Delta H - \sum \alpha_i \Delta F_i= \produitscalaire{\alpha}{B\alpha} \Delta H - \sum_i(b_{ii} \alpha_i +\lambda_i) \Delta F_i.
\]
When $i\in I_{conv}$, we can find    $\lambda_i\ge 0$  such that $b_{ii} \alpha_i +\lambda_i=\alpha_i $
whenever  $b_{ii}\leq 1$. Consequently, the problem reduces to finding
 a  positive semi-definite matrix $B$ of size $m\times m$ such that 
\[
\left\{
\begin{array}{l}
\produitscalaire{e_i}{Be_i}\leq1, \quad \forall i\in I_{conv}\\
\produitscalaire{e_i}{Be_i}=1, \quad \forall i\notin I_{conv}\\
\produitscalaire{\alpha}{B\alpha}=1
\end{array}
\right.
\] 
where $(e_i)_i$ is the canonical basis of $\mathbb{R}^m$.
Equivalently, do there exist $v_1,\ldots, v_m \in \mathbb{R}^m$  such that
\[
\left\{
\begin{array}{l}
\left| v_i \right|\leq1,\quad\forall i\in I_{conv}\\
\left| v_i \right|=1,\quad\forall i\notin I_{conv}\\
\left| \sum\alpha_i v_i \right|=1
\end{array}
\right.
\quad ?
\] 
We conclude with Lemma \ref{lemma:imphiconcave}.
\end{proof}
        
\section{Back to sets}\label{sec:sets}

This sections explains how to complete the proof of Theorem~\ref{main_theorem}. The main issue is to get rid of the 
smoothness assumptions made so far.  The  plan of the argument is summed up in the next figure.
The key point is that the conditions on $\alpha$ do not depend on $n$.
\medskip

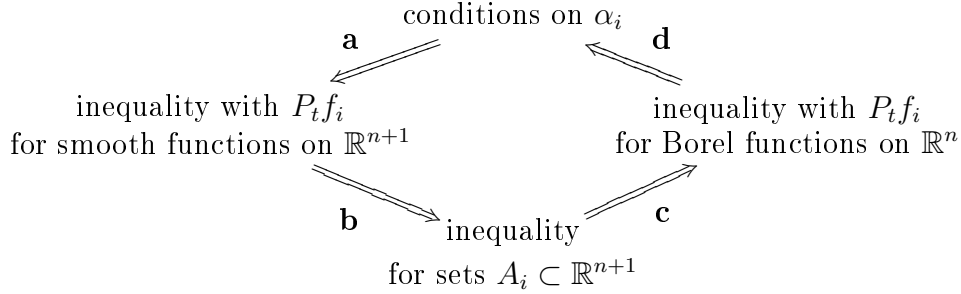
\hfill
\xymatrix@C=.5em@R=1em@M=5pt{
& *+\txt{\phantom{\huge gg}\makebox[0pt]{conditions on $\alpha_i$}\phantom{\huge gg}}
\ar@{=>}[dl]_*+{\txt{\bf a}} & \\
*+\txt{inequality with $P_tf_i$\\ for smooth functions on $\mathbb{R}^{n+1}$}
\ar@{=>}[dr]_*+{\txt{\bf b}}
& &
*+\txt{inequality with $P_tf_i$\\ for Borel functions on $\mathbb{R}^{n}$}
\ar@{=>}[ul]_*+{\txt{\bf d}}\\
& 
*+\txt{\phantom{\huge gg}\makebox[0pt]{inequality}\phantom{\huge gg}\\\makebox[0pt]{for sets $A_i \subset\mathbb{R}^{n+1}$}} 
\ar@{=>}[ur]_*+{\txt{\bf c}} &
}
\hfill\ 
\medskip

If we can prove the above implications, we will have shown that
\[
\text{\em assertion 1} \Longleftrightarrow \text{\em assertion 2}
\Longleftrightarrow \text{\em assertion 4}
\]
in Theorem \ref{main_theorem}. Moreover, it is clear that 
$\text{\em assertion 4}
\Longrightarrow \text{\em assertion 3}$.
To complete the picture, we can for instance prove $\text{\em assertion 3}
\Longrightarrow \text{\em assertion 1}$ in the same way we do below for the fourth implication.
\bigskip
 
\noindent{\em a- ``Conditions on $\alpha_i$ $\Rightarrow$ inequality with $P_tf_i$ for smooth functions on $\mathbb{R}^{n}$'':}
This implication is nothing else than Theorem  \ref{theorem:concave}. Equivalently, the first assertion
in Theorem~\ref{main_theorem} implies the fourth one restricted to ``smooth'' functions (i.e.  verifying all the 
 assumptions of  the first paragraph of Theorem~ \ref{theorem:concave}).

\medskip
\noindent{\em b- ``Inequality with $P_tf_i$ for smooth functions on $\mathbb{R}^{n}$ $\Rightarrow$ inequality for sets $A_i \subset\mathbb{R}^{n}$'':}
For arbitrary $\alpha$, let
 us prove that the fourth assertion in Theorem~\ref{main_theorem} restricted to smooth functions (in the above-mentioned sense)
 implies the second assertion of the theorem, involving sets.
        Let $A_1,\ldots,A_m$ be Borel sets in $\mathbb{R}^n$ with $A_i$ convex when $i\in I_{conv}$. By inner regularity of the measure, we can assume that they are compact. 
        Let $\varepsilon>0$ and $ b>a$ be fixed.
        Then, 
\begin{itemize}
\item for $i\notin I_{conv}$: there exists a smooth function $f_i$ such that $f_i=\Phi(b)$ on $A_i$, $f_i=\Phi(a)$ off $A_i^\varepsilon$, and $0<\Phi(a)\leq f_i\leq \Phi(b)<1$.

\item for $i\in I_{conv}$: there exists a smooth function $f_i$ such that $F_i=\Phi^{-1}\circ f_i$ is concave, $F_i=b$ on $A_i$, $F_i\leq a$ off  $A_i^\varepsilon$, and $F_i \leq b$ on $\mathbb{R}^n$.

For instance, take a point $x_i$ in $A_i$ and define the gauge of $A_i^{\varepsilon/3}$ with respect to $x_i$ by
\begin{equation*}
\rho(x)=\inf\left\{\lambda>0,x_i + \frac{1}{\lambda}(x-x_i) \in A_i^{\varepsilon/3}\right\}.
\end{equation*}
We know that $\rho$ is convex since $A_i$ is convex (see for instance \cite{schnBMT}).
Then set
\begin{equation*}
\tilde{F_i}(x)= b + c \Big( 1- \max\big(\rho(x)\,,\, 1\big)\Big)
\end{equation*}
where $c>0$ is chosen large enough to insure that $\tilde{F_i}\leq a$ off $A_i^{2\varepsilon/3}$.
Now, we can take a smooth function $g$ with compact support small enough and of integral 1, such that 
$f_i=\Phi\big(\tilde{F_i}*g\big)$ is a smooth $\Phi^{-1}$-concave function satisfying the required conditions.

\item for $h$: set 
        \[
        a_0=\max_{\begin{array}{c}u_i= a \text{ or } b \\ u\neq(b,\ldots,b)\end{array}}         
        \sum\alpha_i u_i \qquad \mbox{and} \qquad
        b_0=\sum\alpha_i b.
        \]
        Again, we can choose a smooth function $h$ such  that $h=\Phi(b_0)$ on $\sum\alpha_i A_i^\varepsilon$, $h=\Phi(a_0)$ off $\big(\sum\alpha_i A_i^\varepsilon\big)^\varepsilon$, and $0<\Phi(a_0)\leq h\leq \Phi(b_0)<1$.
\end{itemize}   
        From these definitions, the functions $h$ and $f_i$ are ``smooth'' and satisfy 
        \begin{equation*}
        \forall x_i\in\mathbb{R}^n,\quad \Phi^{-1}\circ h\big({\textstyle\sum}\alpha_i x_i\big)\geq \sum \alpha_i \Phi^{-1}\circ f_i(x_i).
        \end{equation*}
       By our hypothesis, the inequality remains valid with $P_th$ and $P_t f_i$ for all $t> 0$. Choosing $t=1$, $x_i=0$ yields
        \begin{equation*}
        \Phi^{-1}\left( \int h \, d\gamma_n\right) \geq \sum \alpha_i \Phi^{-1}\left( \int f_i \, d\gamma_n\right).
        \end{equation*} 
 
 Remark here that the functions depends actually of $a$ (respectively $a_0$), $b$ (respectively $b_0$), and $\varepsilon$, possibly in a precise way with a procedure like described above for $f_i$. We could then write $h(a_0,b_0,\varepsilon,.)$ and $f_i(a,b,\varepsilon,.)$.
   
   Letting first $a\to-\infty$ so that $a_0\to-\infty$, we get by dominated convergence
   \begin{equation*}
        \Phi^{-1}\left( \int h(-\infty,b_0,\varepsilon,.) \, d\gamma_n\right) \geq \sum \alpha_i \Phi^{-1}\left( \int f_i(-\infty,b,\varepsilon,.) \, d\gamma_n\right).
        \end{equation*} 
   Now let $(b,\varepsilon)$ tend to $(\infty,0)$. Notice that $f_i(-\infty,\infty,0,.)$ and $h(-\infty,\infty,0,.)$ are characteristic functions. Eventually we obtain, again by dominated convergence, that
        \[
        \Phi^{-1}\circ\gamma_n\Big({\textstyle\sum}\alpha_i A_i\Big)\geq  \sum \alpha_i \Phi^{-1}\circ\gamma_n(A_i).
        \]
        
\noindent
{\em c- ``Inequality for sets $A_i \subset\mathbb{R}^{n+1}$ $\Rightarrow$ inequality with $P_tf_i$ for Borel functions on $\mathbb{R}^{n}$''}.  Here we assume that the second assumption
of Theorem~\ref{main_theorem} is valid for all Borel sets in $\mathbb{R}^{n+1}$ and we derive the fourth assumption of the 
theorem for functions defined on $\mathbb{R}^n$.

For any Borel function $f$ on $\mathbb{R}^n$ taking values in $[0,1]$, $t>0$, and $x\in\mathbb{R}^n$, we define 
        \[
        B_f^{t,x}=\Big\{(u,y) \,\big|\, u\leq\Phi^{-1}\circ f\big(x+\sqrt{t} \, y\big) \Big\} \subset\mathbb{R}\times\mathbb{R}^{n}.
        \]
        Then it holds
        \[
        \gamma_{n+1}\big(B_f^{t,x}\big)=P_tf(x).
        \]
Let $h,f_1,\ldots,f_n$ be Borel functions on $\mathbb{R}^n$ with  values in $[0,1]$, such that $\Phi^{-1}\circ f_i$ is concave when $i\in I_{conv}$.
Assume that 
\begin{equation*}
\forall x_i\in\mathbb{R}^n,\quad \Phi^{-1}\circ h\big({\textstyle\sum}\alpha_i x_i\big)\geq \sum \alpha_i \Phi^{-1}\circ f_i(x_i).
\end{equation*}
Then for $(u_i,y_i)$ in $B_{f_i}^{t,x_i}$, we get
        \[
        \sum \alpha_i u_i 
        \leq \sum \alpha_i \Phi^{-1}\circ f_i(x_i+\sqrt{t} \, y_i) 
        \leq \Phi^{-1}\circ h\big({\textstyle\sum} \alpha_i (x_i+\sqrt{t} \, y_i)\big)
        \]
        which means that 
        \[
        \sum \alpha_i B_{f_i}^{t,x_i} \subset B_{h}^{t,{\sum} \alpha_i x_i}.
        \]
The same argument shows that $B_{f}^{t,x}$ is convex if $\Phi^{-1}\circ f$ is concave.
Thus, the result for sets in $\mathbb{R}^{n+1}$ implies that
        \[
        \Phi^{-1}\circ P_th\big({\textstyle\sum}\alpha_i x_i\big)\geq 
        \Phi^{-1}\circ \gamma_{n+1}\left({\textstyle\sum} \alpha_i B_{f_i}^{t,x_i}\right)
        \geq \sum \alpha_i\Phi^{-1}\circ P_t f_i(x_i).
        \]

\medskip
\noindent
{\em d-  ``Inequality with $P_tf_i$ for Borel functions on $\mathbb{R}^{n}$ $\Rightarrow$ conditions on $\alpha_i$'':}
We will prove the contraposed assertion: if the conditions on $\alpha_i$ are violated, then
there exists Borel functions $h$ and $f_i$ such that $\Phi^{-1}\circ f_i$ is concave for $i\in I_{conv}$,
which verify for all $x_i$ the relation $\Phi^{-1}\circ h(\sum \alpha_i x_i)\ge \sum \Phi^{-1}\circ f_i(x_i)$
but for which this inequality is not preserved by $P_t$ for some $t$. Actually since $P_1 f(0)=\int f\, d\gamma$,
it will be enough to exhibit functions such  that
  $$\Phi^{-1}\left(\int h \, d\gamma\right)< \sum \alpha_i \Phi^{-1}\left(\int f_i \, d\gamma\right).$$
Let $f:\mathbb{R}^n\to (0,1)$ be an even   Borel function such that
$$ f(0)>\frac{1}{2}, \quad  \int f  \, d\gamma <\frac12, \quad \mbox{and} \quad F=\Phi^{-1}\circ f \quad \mbox{is concave}.$$
For instance, we may take $f(x)=\Phi\big(1-\left| ax \right|^2\big)$  for $a$ large enough.
Note  that for $0\leq t\leq1$,
\begin{equation}\label{eq:t}
F(tx)\geq tF(x)+(1-t)F(0)\geq tF(x).
\end{equation}

 Assume first that $\sum\alpha_i<1$.
Then by concavity and the latter bound, we get  for all $x_i$,
\begin{align*}
\Phi^{-1}\circ f\big({\textstyle\sum_i\alpha_i x_i}\big)& =F\big({\textstyle\sum_i\alpha_i x_i}\big)
\geq \sum_i \frac{\alpha_i}{\sum_j\alpha_j} F\Big(\big({\textstyle\sum_j\alpha_j}\big)\, x_i\Big)\\
&\geq \sum_i \alpha_i F(x_i)=  \sum_i \alpha_i \Phi^{-1}\circ f(x_i).
\end{align*}
However since $1>\sum \alpha_i$ and  $\Phi^{-1}\big(\int f \, d\gamma\big) <0$, it holds
\begin{equation*}
\Phi^{-1}\left(\int f \, d\gamma\right) < \sum_i \alpha_i \Phi^{-1}\left(\int f \, d\gamma\right).
\end{equation*}

Assume now that  there exists $j\notin I_{conv}$ such that $\alpha_j-\sum_{i\neq j}\alpha_i > 1$.
Then using \eqref{eq:t} and concavity again, we obtain for all $x_i$,
\begin{align*}
 \alpha_j F(x_j) &\ge 
    \left(1+  {\textstyle\sum_{i\neq j}} \alpha_i\right) F\left(\frac{\alpha_j x_j}{1+  \sum_{i\neq j} \alpha_i}\right)\\
  &\ge 
 F\left(\alpha_j x_j -  {\textstyle\sum_{i\neq j}} \alpha_i x_i\right) +  \sum_{i\neq j} \alpha_i F(x_i).
\end{align*}
Let $g=1-f$. Since $-F=-\Phi^{-1}\circ f= \Phi^{-1}\circ(1-f)=\Phi^{-1}\circ g$ and $f$ is even
 we may rewrite the latter as
\begin{equation*}
 \Phi^{-1}\circ g\left(\alpha_j x_j + {\textstyle\sum_{i\neq j}} \alpha_i (-x_i)\right)\geq
 \alpha_j \Phi^{-1}\circ g(x_j) +  \sum_{i\neq j} \alpha_i \Phi^{-1}\circ f(-x_i).
 \end{equation*}
However, since $\Phi^{-1}(\int g\, d\gamma)=-\Phi^{-1}(\int f\, d\gamma)>0$ and $\alpha_j-\sum_{i\neq j}\alpha_i>1$ 
it also holds
 \begin{equation*}
\Phi^{-1}\left(\int g \, d\gamma\right) < \alpha_j \Phi^{-1}\left(\int g \, d\gamma\right) + \sum_{i\neq j} \alpha_i \Phi^{-1}\left(\int f \, d\gamma\right).
\end{equation*}
Therefore the proof is complete.

\section{Further remarks}
\subsection{Brascamp-Lieb type inequalities}
In the previous papers \cite{bore93gpfd,bore00degi}, Borell already used his semigroup approach to
derive variants of the Prékopa-Leindler inequality. The later is a functional counterpart to the Brunn-Minkowski 
inequality for the Lebesgue measure and reads as follows: if $\lambda \in (0,1)$ and $f,g,h:\mathbb R^n\to \mathbb R^+$ are Borel functions such that for all $x,y\in \mathbb R^n$,
$$ h\big(\lambda x + (1-\lambda)y \big)\ge f(x)^\lambda g(y)^{1-\lambda}$$
then $\int h \ge \left(\int f\right)^\lambda \left( \int g\right)^{1-\lambda}$ where the integrals are 
with respect to Lebesgue's  measure. Borell actually showed the following stronger fact:
for all $t>0$ and all $x,y\in \mathbb R^n$
$$ P_t h\big(\lambda x + (1-\lambda)y \big)\ge P_tf(x)^\lambda P_tg(y)^{1-\lambda}.$$
Setting $H(t,\cdot)=\log P_t h$ and defining $F,G$  similarity, it is proved that 
$C(t,x,y):= H\big(t,\lambda x + (1-\lambda)y \big)- \lambda F(t,x)+(1-\lambda) G(t,y)$
satisfies a positivity-preserving evolution equation. The argument is simpler than for Ehrhard's inequality
since the evolution equation of individual functions is simpler: $2\partial_t H= \Delta H +|\nabla H|^2$.

The Brascamp-Lieb \cite{brasl76bcyi,lieb90gkgm} inequality is a powerful extension of  Hölder's inequality. The so-called reverse
Brascamp-Lieb inequality, first proved in \cite{bart97iblc,Barthe-BL}, appears as an extension of the Prékopa-Leindler
inequality. In the paper \cite{bartc04ibli}, it was noted that Borell's semigroup method
could be used to derive the geometric reverse Brascamp-Lieb inequality (which in some sense is a generic 
case, see  \cite{benncct05blif}) for functions of one variable. This observation was also motivated by a proof of the Brascamp-Lieb
inequalities based on semigroup techniques (Carlen Lieb and Loss  \cite{carlll04sayi} for functions of one variable,
and Bennett Carbery Christ and Tao \cite{benncct05blif} for general functions). In this subsection, we take advantage of our
streamlined presentation of Borell's method, and quickly reprove the reverse Brascamp-Lieb inequality
in geometric form, but for functions of several variables. More surprisingly we will recover the 
Brascamp-Lieb from inequalities which are preserved by the Heat flow. The result 
is not new (the inequality for the law of the semigroup appears in the preprint \cite{bartclm06spbl}), but it is 
interesting to have semigroup proofs of the direct and of the reverse inequalities which follow exactly
the same lines. Recall that the transportation argument developed in \cite{Barthe-BL} was providing
the direct and the reverse inequality simultaneously.

The setting of the geometric inequalities is as follows: for $i=1,\ldots,m$ let $c_i>0$ and let $B_i:\mathbb R^N \to \mathbb R^{n_i}$ be linear maps such that $B_iB_i^*=I_{n_i}$ and 
\begin{equation}\label{eq:decomp}
\sum_{i=1}^m c_i B_i^* B_i =I_N.
\end{equation}
These hypotheses were put forward by Ball in connection with volume estimates in convex geometry \cite{ball89vscr}. Note that 
$B_i^*$ is an isometric embedding of $\mathbb R^{n_i}$ into $\mathbb R^N$ and that $B_i^*B_i$ is the orthogonal projection
from $\mathbb R^N$ to $E_i=\mathrm{Im}(B_i^*)$. The Brascamp-Lieb inequality asserts that for all Borel functions  
$f_i:\mathbb R^{n_i}\to \mathbb R^+$ it holds
$$ \int_{\mathbb R^N} \prod_{i=1}^m f_i(B_i x)^{c_i} \, dx \le \prod_{i=1}^m \left(\int_{\mathbb R^{n_i}} f_i\right)^{c_i}.$$
The reverse inequality ensures that 
$$ \int_{\mathbb R^N}^*   \sup\left\{\prod_{i=1}^m f_i(x_i)^{c_i} ;\; x_i\in \mathbb R^{n_i} \mbox{with}
   \sum c_i B_i^*x_i=x\right\} \, dx \ge \prod_{i=1}^m \left(\int_{\mathbb R^{n_i}} f_i\right)^{c_i}.$$

Following \cite{bartc04ibli}, we will deduce the later from the following result. 
\begin{theoreme}\label{th:RBL} If $h:\mathbb R^N \to \mathbb R^+$
and $f_i:\mathbb R^{n_i} \to \mathbb R^+$ satisfy 
\[
\forall x_i\in \mathbb R^{n_i}, \quad 
h\Big(\sum_{i=1}^m c_i B_i^* x_i\Big)\ge \prod_{i=1}^m f_i(x_i)^{c_i}
\]
then 
\[
\forall x_i\in \mathbb R^{n_i}, \quad
P_th\Big(\sum_{i=1}^m c_i B_i^* x_i\Big)\ge \prod_{i=1}^m P_tf_i(x_i)^{c_i}.
\]
\end{theoreme}
The reverse inequality is obtained as $t\to +\infty$ since for $f$ on $\mathbb R^d$, $P_tf(x)$ is equivalent to $(2\pi t)^{-d/2} \int_{\mathbb R^d} f$. To see it, note that:
\[
P_tf(x)= (2\pi t)^{-d/2} \int_{\mathbb R^d} f(y) \exp\left(\frac{|{x-y}|^2}{2t}\right)\, dy.
\]
 Note also that taking traces in the decomposition of the identity map yields 
$\sum_i c_i n_i=N$.

In order to recover the Brascamp-Lieb inequality, we will show the following theorem.
\begin{theoreme}\label{th:BL} If $h:\mathbb R^N \to \mathbb R^+$
and $f_i:\mathbb R^{n_i} \to \mathbb R^+$ satisfy 
\[
\forall x \in \mathbb R^N, \quad
h(x)\le \prod_{i=1}^m f_i(B_i x)^{c_i},
\]
then 
\[
\forall x \in \mathbb R^N, \quad
P_th(x)\le \prod_{i=1}^m P_tf_i(B_i x)^{c_i}.
\]
\end{theoreme}
Again, the limit $t\to +\infty$ yields the Brascamp-Lieb inequality when choosing $h(x)=\prod_{i=1}^m f_i(B_i x)^{c_i}$.
We sketch the proofs the the above two statements, omitting the truncation arguments needed to ensure
Condition  \eqref{infinipositif}.
\begin{proof}[Proof of Theorem~\ref{th:RBL}]
Set $H(t,\cdot)=\log P_th(\cdot)$ and $F_i(t,\cdot)=\log P_tf_i(\cdot)$. As said above, the functions $H$ and $F_i$ satisfy the equation 
$2\partial_t U =\Delta U+|\nabla U|^2$. Set for $(t,x_1,\ldots,x_m)\in \mathbb R^+\times \mathbb R^{n_1}\times
\cdots \times \mathbb R^{n_m}$
$$C(t,x_1,\ldots,x_m):= H\Big(t, \sum_{i=1}^m c_i B_i^* x_i\Big)-\sum_{i=1}^m c_i F_i(t,x_i).$$
By hypothesis $C(0,\cdot)\ge 0$ and we want to prove that $C(t,\cdot)$ is non-negative as well.
As before, we are done if we can show that the three conditions 
$C\le 0$, $ \nabla C=0$, and $\mathrm{Hess}( C)\ge 0$ imply 
that $\partial_t C\ge 0$. Actually one can see that  the condition $C\le 0$ will not be used in the following. 
Omitting variables, 
$$ 2 \partial_t C= \left( \Delta H-\sum c_i \Delta F_i\right)  + \left(|\nabla H|^2-\sum c_i |\nabla F_i|^2\right)=:
\mathcal S+\mathcal P.$$
Straightforward calculations give
\begin{align*}
 \nabla_{x_i} C&= c_i B_i \nabla H -c_i \nabla F_i \quad \mbox{and} \\
  \mathrm{Hess}_{x_i,x_j} (C)&=  c_ic_j B_i \mathrm{Hess}(H) B_j^*-\delta_{i,j} c_i \mathrm{Hess }(F_i).
\end{align*}
Note that the decomposition \eqref{eq:decomp} implies for all $v\in \mathbb R^N$
$$ |v|^2= \produitscalaire{v}{\sum c_i B_i^*B_i v}= \sum c_i |B_i v|^2.$$
Hence, if $\nabla C=0$, the above calculation gives $\nabla F_i =B_i \nabla H$. Consequently
$|\nabla H|^2=\sum c_i |B_i \nabla H|^2=\sum c_i |\nabla F_i|^2$. So $\nabla C=0 \, \Longrightarrow \mathcal P=0$.

Next, we deal with the second order term. Using \eqref{eq:decomp} again
\begin{align*}
 \Delta H &= \mathrm{Tr} \big(\mathrm{Hess}( H)\big) 
     =  \mathrm{Tr} \Big(\big(\sum_i c_i B_i^*B_i\big)\mathrm{Hess}( H) \big(\sum_j c_j B_j^*B_j\big)\Big)\\
      &= \sum_{i,j} \mathrm{Tr} \Big(B_i^*\big(  c_ic_j  B_i \mathrm{Hess}(H)   B_j^*\big)B_j\Big).
\end{align*}
Also note that 
\begin{align*}
   \sum_{i,j} \mathrm{Tr} \Big(B_i^*\big( \delta_{i,j} c_i   \mathrm{Hess} (F_i)\big)B_j\Big)
    &= \sum_i \mathrm{Tr} \big(B_i^*  c_i   \mathrm{Hess} (F_i) B_i\big)\\
    &=
     \sum_i  c_i \mathrm{Tr} \big(    \mathrm{Hess} (F_i) B_iB_i^*\big)
     =\sum_i  c_i \Delta F_i,
\end{align*}
since $B_iB_i^*=I_{n_i}$.
Combining the former and the later and denoting by $J_i$ the canonical embedding of $\mathbb R^{n_i}$ into
$\mathbb R^{n_1+\cdots+n_m}$ we get that
\begin{align*}
\mathcal S &= \Delta H -\sum c_i \Delta F_i = \sum_{i,j} \mathrm{Tr} \big(B_i^* \mathrm{Hess}_{x_i,x_j}( C)  B_j\big)\\
  &= \sum_{i,j} \mathrm{Tr} \Big(B_i^*\big( J_i^*\mathrm{Hess}(C) J_j\big)B_j\Big)
  =
   \mathrm{Tr} \Big(\big(\sum_{i} J_i  B_i\big)^* \mathrm{Hess}(C) \big(\sum_j J_j B_j\big)\Big)
\end{align*}
is non-negative when $\mathrm{Hess}(C)\ge 0$. This is enough to conclude that $C$ remains non-negative.
\end{proof}
\begin{proof}[Proof of Theorem~\ref{th:BL}]
As before we set   $H(t,\cdot)=\log P_th(\cdot)$ and $F_i(t,\cdot)=\log P_tf_i(\cdot)$. For $(t,x)\in \mathbb R^+\times \mathbb R^{N} $
$$C(t,x ):=\sum_{i=1}^m c_i F_i(t,B_i x)- H (t,  x).$$
Omitting variables, $C$ evolves according to the equation
$$ \partial_t C= \left( \sum c_i \Delta F_i -\Delta H\right)+  \left( \sum c_i |\nabla F_i|^2-|\nabla H|^2\right)=:\mathcal S+\mathcal P.$$
Next 
$$ \nabla C= \sum c_i B_i^* \nabla F_i -\nabla H \quad \mbox{and} \quad \mathrm{Hess}( C)= \sum c_i B_i^* \mathrm{Hess} (F_i) B_i- \mathrm{Hess}( H).$$
Taking traces in the later equality and since $B_iB_i^*=I_{n_i}$ we obtain
$$ \Delta C= \sum_i c_i \mathrm{Tr}\Big( \mathrm{Hess}( F_i) B_iB_i^*\Big)-\Delta H=\sum_i c_i \Delta F_i -\Delta H=\mathcal S.$$
Therefore the second order term is clearly elliptic. 

It remains to check that $\nabla C=0$ implies that the first order term $\mathcal P$ is non-negative.
We will need the following easy consequence of the decomposition \eqref{eq:decomp}: if $x_i\in \mathbb R^{n_i}$,
$i=1,\ldots,m$, then $$\Big|\sum c_i B_i^* x_i\Big|^2 \le \sum c_i |x_i|^2.$$
The proof is easy: set $v=\sum c_i B_i^* x_i$. Then by Cauchy-Schwarz
\begin{align*}
  |v|^2&= \produitscalaire{v}{\sum c_i B_i^* x_i} =\sum c_i  \produitscalaire{B_i v}{x_i} \\
  &\le \Big( \sum c_i |B_i v|^2\Big)^{\frac{1}{2}}\Big(\sum c_i |x_i|^2 \Big)^{\frac{1}{2}}.
\end{align*}
But \eqref{eq:decomp} ensures that $|v|^2= \sum c_i |B_i v|^2$ so after simplification we get the claim.
Finally, note that  $\nabla C=0$ means that  $\nabla H=\sum c_i B_i^* \nabla F_i$. Hence
$|\nabla H|^2\le \sum c_i|\nabla F_i|^2$. In other words $\mathcal P\ge 0$. The proof is therefore complete.
\end{proof}

\subsection{Looking for  Gaussian Brascamp-Lieb inequalities}
It is natural to ask about Gaussian versions of the Brascamp-Lieb or inverse Brascamp-Lieb inequalities.
For $0\leq i\leq m$, take a nonzero real $d_i$, a positive integer $n_i\leq N$, a linear surjective map $L_i : \mathbb{R}^N\to\mathbb{R}^{n_i}$, and a Borel function $f_i$ on $\mathbb{R}^{n_i}$ taking value in $(0,1)$. Does the inequality
\[
\forall x \in \mathbb{R}^N,\quad \sum_{i=0}^{m} d_i \Phi^{-1}\circ  f_i (L_ix)\ge 0
\]
upgrade for all $t\geq 0$ to
\[
 \forall x \in \mathbb{R}^N,\quad \sum_{i=0}^{m} d_i \Phi^{-1}\circ  P_t f_i (L_ix)\ge 0 \quad?
\]
This general formulation allows negative $d_i$'s and would encompass Gaussian extensions of Theorem~\ref{th:RBL} or  Theorem~\ref{th:BL}. It also enables  a better understanding of the essential properties in the semigroup argument. Note that from now the index $i$ goes from 0 to $m$, the function $f_0=:h$ playing \emph{a priori} no particular role anymore.

As before, we define for  $t\geq 0$ and $x\in\mathbb{R}^N$, 
\[
C(t,x)=\sum d_i \Phi^{-1}\circ P_t f_i (L_ix)=\sum d_i F_i(t,L_i x)
\]
and we are interested in proving that 
$
C(0,\,.\,) \geq 0$ implies $ C(t,\,.\,) \geq 0 $  for all $ t\ge0$.
Assume that our functions are smooth enough for the next calculations. It holds
\[
\begin{array}{l@{\,}c@{\,}>{\displaystyle}l}
C &=& \sum d_i  F_i,\\
\nabla C &=& \sum d_i L_i^* \nabla F_i, \\
\mathrm{Hess\,}(C) &=& \sum d_i L_i^* \mathrm{Hess\,}(F_i) L_i,\\
\end{array}
\]
and thanks to the Heat equation, $C$ satisfies the following differential equation 
$
2 \partial_tC=  (\mathcal{S}  + \mathcal{P} )
$
where  
\[
\mathcal{S} =\sum d_i \Delta F_i
\mbox{ and }
 \mathcal{P}  =- \sum d_i \left| \nabla F_i \right|^2 F_i.
\]
We require that \label{gbl:requirement}
\begin{equation*}
\left\{
\begin{array}{l}
\mathrm{Hess} (C) \geq 0\\
\nabla C=0\\
C\leq 0
\end{array}
\right.
\quad\Longrightarrow \quad
\left\{
\begin{array}{l}
\mathcal{P}\geq0\\
\mathcal{S}\geq0
\end{array}
\right.\end{equation*}
in order to apply Lemma \ref{lemma:equadiff} (the condition at infinity is verified, provided one restricts to good 
enough functions $f_i$. We omit the details). This request will translate in terms of conditions on the data $(d_i,L_i)$.
We deal separately with the condition for each order:

\begin{description}
\item[First order terms :] note that $(F_i,\nabla F_i)_{i=0,\ldots,m}$ can be chosen arbitrarily for fixed $x$ and $t$; for instance take 
\[
f_i: x_i'\mapsto \produitscalaire{\Phi'(Z_i)Y_i}{x_i'} + \Phi(Z_i) - P_t \tilde{f_i}(L_ix)
\]
with $\tilde{f_i}:x_i'\mapsto \produitscalaire{\Phi'(Z_i)Y_i}{x_i'}$,
so that $F_i(t,L_i x)=Z_i$ and $\nabla F_i(t,L_i x)=Y_i$.

Thus
the condition $(C\le 0,\, \nabla C=0)\Longrightarrow
\mathcal P\ge 0$ boils down to the following relation between polynomials
\[
\left\{\begin{array}{ccc}
\sum d_i Z_i & \leq&0\\
\sum d_i L_i^* Y_i&=&0
\end{array}
\right.
\quad\Longrightarrow \quad \sum d_i \left| Y_i \right|^2 Z_i  \leq 0
\]
where $Z_i$ is a $1-$dimensional unknown and $Y_i$ is an $n_i-$dimensional one. 

Reasoning for fixed $Y_i's$, and viewing the conditions on $Z_i$ as equations of half-spaces, we easily see that the 
later condition is equivalent to
\begin{equation}\label{eq:restriction}
\sum d_i L_i^* Y_i=0 \quad\Longrightarrow \quad \left| Y_0 \right|_{\mathbb{R}^{n_0}}^2=\ldots=\left| Y_m \right|_{\mathbb{R}^{n_m}}^2.
\end{equation}
This condition can be worked out a bit more. Let $\mathcal L:\mathbb R^{\sum n_j}\to \mathbb R^N$ be  defined by 
$$\mathcal L(Y_0,\ldots,Y_m)=\sum d_i L_i^* Y_i.$$
 If $a=(a_0, \ldots,a_m)$ and $b=(b_0, \ldots,b_m)$ belong to $\ker\mathcal L$
then $|a_i|^2$, $|b_i|^2$, and by linearity $|a_i+b_i|^2$ are independent of $i$. Expanding the square of the sum, we deduce 
that $ \produitscalaire{a_i}{b_i}$ is independent of $i$ and therefore equal to the average over $i$ of these quantities.
Hence for all $i$, $(m+1)  \produitscalaire{a_i}{b_i}= \produitscalaire{a}{b}$. This means that $u_i:\ker\mathcal L\to \mathbb R^{n_i}$
defined by $u_i(a)=\sqrt{m+1}\, a_i$ is an isometry. Since $a_i=u_i\big(u_0^{-1}(a_0)\big)$, we conclude that 
$$\ker\mathcal L =\left\{ \Big(a_0,u_1\big(u_0^{-1}(a_0)\big),\ldots,u_m\big(u_0^{-1}(a_0)\big)\Big);\; a_0\in \mbox{Im}(u_0) \right\}.$$
It is then clear that Condition \eqref{eq:restriction} is equivalent to the following: there exists a subspace $X\subset \mathbb R^{n_0}$
and linear isometries $R_i:X\to\mathbb R^{n_i}$, $i\ge 1$ such that 
\begin{equation}
   \label{eq:restriction2}
   \ker\mathcal L =\big\{ (x,R_1x,\ldots,R_mx);\; x\in X\big\}.
\end{equation}

\item [Second order terms :] we are done if we can find  an elliptic operator $\mathcal{E}$ such that $\mathcal{S}=\mathcal{E}C$.
 In other words we are looking for   a symmetric positive semi-definite matrix  $A$ of size $N\times N$ such that the quantity
$$  \mathrm{Tr\,} \big(A\, \mathrm{Hess} (C)  \big)= \sum d_i \mathrm{Tr\,} \big(A L_i^* \mathrm{Hess} (F_i) L_i  \big)$$
coincides with $\mathcal{S} =\sum d_i \Delta F_i$.
As we require this identity for arbitrary functions  $F_i$,  we can conclude that $A$ does the job if and only if 
 for all $0\leq i\leq m$,
\[L_i A L_i^* = I_{n_i}.\]
Eventually, we may look for $A$ in the form $A=\sigma^* \sigma$ for some square matrix $\sigma$ of size $N$.
For $0\le i\le m$ and $1\le j\le n_i$, denote by $u_i^j\in \mathbb{R}^N$ the columns of $L_i^*$.
Rewriting the later conditions in terms of $\sigma$ we may conclude that: $\mathrm{Hess}(C)\ge 0\, \Longrightarrow \mathcal S \ge 0$
holds provided  there exits a matrix $\sigma$ of size $N$ such that for all $0\le i\le m$ the vectors $(\sigma u_i^j)_{j=1}^{n_i}$
form an orthonormal system in $\mathbb R^N$.
Note that the first order condition requires that the linear relations between the vector $u_i^j$ should have a particular
structure.
\end{description}

The above conditions are quite restrictive. We were able to find data $(d_i,L_i)$ verifying them, but all of them 
could be reduced to the Borell theorem, using the rotation invariance of the Gaussian measure and the fact that
its marginals remain Gaussian. To conclude this section let us briefly explain  why the method does not allow
any new Gaussian improvement of Theorems~\ref{th:RBL} or  \ref{th:BL}.

For $i=1,\ldots, m$, let $c_i>0$ and $B_i:\mathbb R^n \to \mathbb R^{n_i}$ be linear surjective maps.
If we look for Gaussian versions of the Brascamp-Lieb inequality, we are led to apply the previous 
reasoning to $N=n$, $B_0=I_N$, $d_0=-1$, and for $i\ge 1$, $L_i=B_i$ and $d_i=c_i$.
Now, with the above notation, $(Y_0,\ldots,Y_m)\in \ker \mathcal L$ is equivalent to $Y_0=\sum_{i=1}^m c_i B_i^* Y_i$.
Since this condition can be verified even though $|Y_1|\neq |Y_2|$ we conclude that the first order condition
is never satisfied.

Next, we are looking for inequalities of the reverse Brascamp-Lieb type. Hence we choose
$N=n_1+\cdots+n_m$, $d_0=1$, $L_0(x_1,\ldots,x_m)=\sum c_i B_i^* x_i$, and for $i\ge 1$, $d_i=-c_i$, $L_i(x_1,\ldots,x_m)=x_i$.
For $x\in \mathbb R^n$, $L_0^*(x)=(c_1B_1x,\ldots,c_mB_m x)$. For $i\ge 1$ and $x_i\in \mathbb R^{n_i}$, $L_i^*(x_i)=(0,\ldots,0,x_i,
0, \ldots,0)$ where $x_i$ appears at the $i$-th place. The condition $(Y_0,\ldots,Y_m)\in \ker \mathcal L$, 
that is $L_0^*(Y_0)=\sum_{i\ge 1} c_i L_i^*(Y_i)$ becomes:
  $$ \forall i=1,\ldots, m,\; Y_i=B_i Y_0.$$
Hence $\ker\mathcal L=\big\{ (Y_0, B_1 Y_0,\ldots, B_m Y_0);\;  Y_0\in \mathbb R^n\big\}$. So the first order condition 
\eqref{eq:restriction2}
is verified only if the $B_i$'s are isometries. This forces $n_i=n$ and up to an isometric change of variables, we are
back to the setting of the Gaussian Brunn-Minkowski inequality.

\begin{remarque}
To make use of Lemma \ref{lemma:equadiff}, it should be sufficient to prove $(\mathcal{P}+\mathcal{S}\ge 0)$ instead of the stronger condition $(\mathcal{P}\ge0 \mbox{ and } \mathcal{S}\ge0)$ as required page \pageref{gbl:requirement}. However we were not able to translate this into nice conditions on coefficients or functions. In this sense, our semi-group approach fails to extend Theorem \ref{main_theorem} into a more general Gaussian Brascamp-Lieb inequality.
\end{remarque}

\makeatletter\@ifundefined{url}{\def\url#1{\texttt{#1}}}{}\makeatother

\bigskip

\noindent F. BARTHE:
Institut de Math\'ematiques de Toulouse.  Universit\'e Paul Sabatier. 31062 Toulouse, FRANCE.
Email: barthe@math.univ-toulouse.fr

\medskip
\noindent N. HUET:
Institut de Math\'ematiques de Toulouse.  Universit\'e Paul Sabatier. 31062 Toulouse, FRANCE.
Email: huet@math.univ-toulouse.fr
\end{document}